\documentclass[preprint,12pt]{elsarticle}

\usepackage{graphicx}
\usepackage{amssymb}
\usepackage{epsf}
\usepackage{amsbsy,amsmath}
\usepackage{mathtools}
\usepackage{mathrsfs}
\usepackage{amsfonts}
\usepackage{amssymb}
\usepackage{enumitem}
\usepackage{eucal}
\usepackage{graphics,mathrsfs}
\usepackage{amsthm}
\usepackage{secdot}
\usepackage{esint}
\usepackage{hyperref}
\usepackage{varwidth}
\usepackage{tasks}
\usepackage{geometry}
\newtheorem{theorem}{Theorem}[section]
\newtheorem{lemma}[theorem]{Lemma}

\theoremstyle{definition}
\newtheorem{definition}[theorem]{Definition}

\theoremstyle{remark}
\newtheorem{remark}[theorem]{Remark}
\numberwithin{equation}{section}

\numberwithin{equation}{section}

\usepackage{xcolor}

\renewcommand{\d}{\:\! \mathrm{d}}

\journal{'Complex Variables and Elliptic Equations'}

\begin{document}

\begin{frontmatter}



 \title{Infinitely many small solutions to an elliptic PDE of variable exponent with a singular nonlinearity}
 \author{Sekhar Ghosh\fnref{label2}}
 \ead{sekharghosh1234@gmail.com}
 \author{Debajyoti Choudhuri\fnref{label2}}
 \ead{dc.iit12@gmail.com}
 \author{Ratan Kr. Giri\corref{cor1}\fnref{label3}}
%
 \cortext[cor1]{Corresponding author: {\sffamily giri90ratan@gmail.com/giri@campus.technion.ac.il} (R.Kr. Giri)}
 \fntext[label2]{Department of Mathematics, National Institute of Technology Rourkela, India}
\fntext[label3]{Department of Mathematics, Technion - Israel Institute of Technology, Haifa, Israel}




\begin{abstract}
We prove the existence of infinitely many nonnegative solutions to the following nonlocal elliptic partial differential equation involving singularities
\begin{align}
		(-\Delta)_{p(\cdot)}^{s} u&=\frac{\lambda}{|u|^{\gamma(x)-1}u}+f(x,u)~\text{in}~\Omega,\nonumber\\
	u&=0~\text{in}~\mathbb{R}^N\setminus\Omega,\nonumber
\end{align}
where $\Omega\subset\mathbb{R}^N,\, N\geq2$ is a smooth, bounded domain, $\lambda>0$, $s\in (0,1)$, $\gamma(x)\in(0,1)$ for all $x\in\bar{\Omega}$, $N>sp(x,y)$ for all $(x,y)\in\bar{\Omega}\times\bar{\Omega}$ and $(-\Delta)_{p(\cdot)}^{s}$ is the fractional $p(\cdot)$-Laplacian operator with variable exponent. The nonlinear function $f$ satisfies certain growth conditions. Moreover, we establish a uniform $L^{\infty}(\bar{\Omega})$ estimate of the solution(s) by the Moser iteration technique.
\end{abstract}


\end{frontmatter}


\section{Introduction}
\noindent For many a generation of researchers in the field of analysis of elliptic Partial Differential Equations (PDEs) the finding due to \citet{Ambrosetti1973} still a cornerstone in the study of infinitely many solutions. The authors in \cite{Ambrosetti1973} considered the Laplacian operator ($-\Delta$) involving a superlinear data $f$ with the Dirichlet boundary condition in a bounded, Lipschitz domain $\Omega\subset\mathbb{R}^N$ and have used the symmetric mountain pass theorem due to Clark \cite{Clark1972} to conclude the existence. Mathematically the problem addressed in \citet{Ambrosetti1973} is as follows. 
\begin{align*}
\begin{split}
Lu\equiv \Sigma_{i,j=1}^{N}\frac{\partial }{\partial x_j}\left(a_{ij}(x)\frac{\partial }{\partial x_j}u\right)&=p(x,u)~\text{for}~x\in\Omega ~\\
u|_{\partial\Omega}&=0.
\end{split}
\end{align*}
This led to a chain of works that can be found in \cite{Kajikiya2005,Zhou2014,Gu2018,Gu2018a} and the references therein. Meanwhile, another noteworthy development took place due to \citet{Lazer1991} in studying elliptic partial differential equations involving a singularity. The problem they considered is 
\begin{align*}
\begin{split}
-\Delta u &=p(x)u^{-\gamma}~\text{for}~x\in\Omega, ~\\
u&>0~\text{for}~x\in\Omega, ~\\
u|_{\partial\Omega}&=0,
\end{split}
\end{align*}
where $p(x)>0$ in $\overline{\Omega}$ and $\Omega$ is a bounded domain of sufficiently smooth boundary. PDEs involving singularity further gained popularity amongst the researchers with the advent of nonlocal/local operators as well, viz. \cite{Canino2017,Boccardo2009}. Off late the study of existence and multiplicity (finitely many) of solutions have been explored widely for both the local and non local operators. Due to the vast amount of literature it is almost impossible to assimilate all of them here. What we can do is to direct the reader's attention to \cite{Giacomoni2007,Giacomoni2009,Saoudi2019,Mukherjee2016,Saoudi2017, Servadei2012,Servadei2013} and the references therein. The work due to \cite{Ghosh2019} is the first of its kind that guarantees the existence of infinitely many positive solutions to a fractional Laplacian problem involving both a singularity and a power nonlinearity with constant exponent. The authors in \cite{Ghosh2019} considered the following problem.
\begin{align*}
\begin{split}
(-\Delta)^s u &=\frac{\lambda}{u^{\gamma}}+f(x,u) ~\text{for}~x\in\Omega,~\\
u&>0~\text{for}~x\in\Omega, ~\\
u &=0~\text{for}~x\in\mathbb{R}^N\setminus\Omega,
\end{split}
\end{align*}
where $\Omega$ is a bounded domain in $\mathbb{R}^N$ with a Lipshitz boundary. 
 Recently, PDEs involving operators and nonlinearities of variable orders and exponents respectively took the researchers in a new direction. It is worth mentioning that the literature for the variable order fractional $p(x)$-Laplacian is meager even for the case $p(x)=p$, $1<p<\infty$. In \citet{Fan2007} the author has introduced with the sub-super solution method for the $p(x)$-Laplacian equations. The equivalence of the local minimizers in the $C^1$ topology and the local minimizers in the $W^{1,p(x)}$ topology has been obtained. Further, the author obtained two solutions for an eigenvalue type problem.
 In \cite{Chabrowski2005} the authors have applied the mountain pass theorem to obtain the existence of solutions to the following $p(x)$-Laplacian problem.
 \begin{eqnarray}\label{chab}-div(a(x)|\nabla u|^{p(x)-2}\nabla u)+b(x)|u|^{p(x)-2}u&=&f(x,u)~ \text{in}~ \Omega,\nonumber\\
 u&=&0~ \text{on}~\partial\Omega, 
 \end{eqnarray}
 where $1<p^-\leq p(x)\leq p^+<N$, $\Omega$ is a bounded domain. Here $f$ could be a superlinear or a sublinear function. The readers may also refer to \cite{Fan2003,Bonanno2012} and the references therein for a study of infinitely many solutiuons. It will be a good read to refer \cite{Ji2012} where no Ambrosetti-Rabinowitz type condition has been assumed to obtain multiple solutions to the $p(x)$-Laplacian equations. Some other noteworthy results are due to \cite{Radulescu2015,Tan2012,Bin2014,Zang2008} and the references therein. One may also refer to  \cite{Gao2010,Mihailescu2007,Chung2013,Yao2008} where in the authors have considered a problem with $p(x)$-Laplacian  operator similar to \eqref{chab} with concave-convex nonlinearities.\\ 
 Off late, the author in \cite{Bahrouni2018} studied the new fractional Sobolev space $W^{s,q(x),p(x,y)}(\Omega)$, for $s\in(0,1)$, and the  fractional $p(x)$-Laplace operator. Existence of a solution has been guaranteed using a sub-super solution method and a comparison principle involving the fractional $p(x)$-Laplacian. In \cite{Bahrouni2018a} the authors have developed some qualitative properties of the fractional
Sobolev space $W^{s,q(x),p(x,y)}(\Omega)$ for $s\in(0,1)$. Moreover the authos have studied the following problem \begin{eqnarray}Lu(x) + |u(x)|^{q(x)-1}u(x)& = &\lambda|u(x)|^{r(x)-1}u(x)~\text{in}~ \Omega,\nonumber\\
u(x)& = &0~\text{ in}~ \mathbb{R}^N\setminus\Omega,\nonumber\end{eqnarray}
where $ \lambda > 0, 1 < r(x) < p^-$. One important work which is worth mentioning is due to  \cite{Azroul2019} where the authors have studied the eigenvalue problem involving a fractional $p(x)$-Laplacian. The existence of the eigenvalues and eigenfunctions is based on Ekeland's variational principle. 
 In \cite{Chung2020} the authors obtained the multiplicity of solutions for a problem involving a fractional $p(x)$-Laplacian eoparator. Furthermore, the existence of infinitely many solutions to a similar problem can be found in \cite{Lee2020}. \\It is note worthy here to mention that the fractional order derivatives with variable exponents plays an important role to the nonlinear diffusion process \cite{Lorenzo2002} when it reacts to the temperature change in particular. Another application is due to Chen et al. \cite{Chen2006} to the image restoration. It can be seen from the work due to \cite{Antontsev2005} that the operators with a  variable exponent appear in a natural way in continuum mechanics. Further an application of such problems can also bee seen in the elasticity theory \cite{Zhikov1987}. Physicists also appeal to problems involving fractional order Laplacian of variable exponent in their study of fractional quantum mechanics and Levy's path integral \cite{Laskin2000}.\\ Readers can gain a good amount of information about the fractional Laplacian operator from \cite{Kwasnicki2017,Nezza2012}. The literature survey on problems involving variable order fractional Laplacian and singularity of variable exponent, as far as we know, is almost next to nothing since it is still in the stage of developement. However, motivated from the above mentioned studies and due to the growing interest, we have considered our problem \eqref{prob1} with variable order fractional Laplacian operator involving both singularity and power nonlinearity of variable exponent.
\section{Preliminaries and Main result}
\noindent The aim of this paper is to study the following singular elliptic partial differential equation
\begin{align*}\label{prob1}	
(-\Delta)_{p(\cdot)}^{s} u&=\frac{\lambda}{|u|^{\gamma(x)-1}u}+f(x,u)~\text{in}~\Omega,\\
u&=0~\text{in}~\mathbb{R}^N\setminus\Omega,\tag{P}
\end{align*}
where $\lambda>0$, $\Omega\subset\mathbb{R}^N,\,N\geq2$ is a smooth, bounded domain and $s\in(0,1)$ such that $N>sp(x,y)$ for every $(x,y)\in\bar{\Omega}\times\bar{\Omega}$. The variable exponent function $p(\cdot)$ satisfy the following assumptions.
\begin{flalign*}\label{assm p}
(i)~&p\in C(\bar{\Omega}\times\bar{\Omega}, \mathbb{R}).&\\
(ii)~& p~\text{is symmetric, i.e.}~p(x,y)=p(y,x),\,\forall(x,y)\in\bar{\Omega}\times\bar{\Omega}.\tag{${\mathscr{P}}$}&\\
(iii)~& 1<p^-=\min\limits_{(x,y)\in\bar{\Omega}\times\bar{\Omega}} p(x,y)\leq p(x,y)\leq p^+=\max\limits_{(x,y)\in\bar{\Omega}\times\bar{\Omega}} p(x,y)<\infty.
\end{flalign*}
The fractional $p(\cdot)$-Laplacian operator, $(-\Delta)_{p(\cdot)}^{s}$ (see \cite{Bahrouni2018a}), is defined as
\begin{equation*}
(-\Delta)_{p(\cdot)}^{s}u(x):=P.V.\int_{\mathbb{R}^n}\frac{|u(x)-u(y)|^{p(x,y)-2}(u(x)-u(y))}{|x-y|^{N+sp(x,y)}}\d y~\text{for all}~x\in\mathbb{R}^N,
\end{equation*}
where P.V. refers to the Cauchy principal value of an integral. Note that the usual fractional $p$-Laplacian is a special case of this operator with $p(x,y)\equiv$ constant. The singular exponent $\gamma$ satisfies
\begin{flalign*}\label{assm gamma}
(i)~&\gamma\in C(\bar{\Omega},(0,1))~\text{and}~ 0<\gamma^-=\min\limits_{x\in\bar{\Omega}}\gamma(x)\leq\gamma(x)\leq\gamma^+=\max\limits_{x\in\bar{\Omega}}\gamma(x)<1.\tag{${\mathscr{G}}$}&
\end{flalign*}
The function $f$ satisfies the following $p(\cdot)$-sublinear growth conditions.
\begin{align*}\label{assm f}
(i)~&f\in C(\bar{\Omega}\times\mathbb{R}, \mathbb{R})&\\
(ii)~& \exists\, \delta>0 ~\text{such that}~ \forall\,x\in\Omega ~\text{and}~ |t|\leq\delta,  f(\cdot,t) ~\text{is odd and}~\lim\limits_{t\rightarrow0}\frac{f(x,t)}{t^{(p^--1)}}=+\infty&\\
 & ~\text{uniformly on}~ \Omega \tag{${\mathscr{F}}$}&\\
(iii)~& \exists\, \delta'>0 ~\text{and}~ \alpha(x)\in(1-\gamma(x), p^-) ~\text{with}~ 1-\gamma^-<\alpha^-\leq\alpha^+<p^- ~\text{such that}~ \forall\,x\in\Omega&\\ &~\text{and}~ |t|\leq \delta',~\text{we have}~tf(x,t)\leq\alpha(x) F(x,t), ~\text{where}~ \alpha\in C(\bar{\Omega}, \mathbb{R}) ~\text{and}~ F(x,t)=\int_{0}^{t}f(x,\tau) \d\tau.&\\
(iv)~& 1<\alpha^-=\min\limits_{x\in\bar{\Omega}} \alpha(x)\leq\alpha(x)\leq\alpha^+=\max\limits_{x\in\bar{\Omega}}\alpha(x)<p^-.&
\end{align*}
It is noteworthy that there are no restrictions on $f$, for $t$ at infinity. The presence of the variable exponent in the problem \ref{prob1} naturally confirms that we will look for solutions in a fractional Sobolev space with variable exponent. We first recall some definitions and notations of variable order Lebesgue spaces due to \citet{repos2015}. Let $\Omega \subset\mathbb{R}^N$ be a smooth and open domain. Consider the family $$C_{+}(\bar{\Omega})=\{q\in C(\bar{\Omega},\mathbb{R}):q(x)>1,\,\forall\,x\in\bar{\Omega}\}.$$
Define $q^-=\inf\limits_{x\in{\Omega}}q(x)$ and $q^+=\sup\limits_{x\in{\Omega}}q(x)$ for all $q\in C_{+}(\bar{\Omega})$. For all $q\in C_{+}(\bar{\Omega})$, the Lebesgue space of variable order is defined by
$$L^{q(x)}(\Omega)=\{u: u ~\text{is measurable and}~ \int_{\Omega}|u(x)|^{q(x)}\d x<\infty\}$$
endowed with the Luxemburg norm
$$\|u\|_{q(x)}=\inf\{\mu>0:\int_{\Omega}\left|\frac{u(x)}{\mu}\right|^{q(x)}\!\!\!\!\!\!\d x\leq1\}.$$
The space ($L^{q(x)}(\Omega), \|\cdot\|$) is a Banach space. Moreover, if $\Omega$ is bounded and the following assumption \eqref{assm q} holds, then the space is separable, reflexive, uniformly convex Banach space (\cite[Theorem 1.6, 1.10]{Fan2001}).
\begin{flalign*}\label{assm q}
(i)~& 1<q^-=\min\limits_{x\in\bar{\Omega}}q(x)\leq q(x)\leq q^+=\max\limits_{x\in\bar{\Omega}}q(x)<\infty.\tag{${\mathscr{Q}}$}&
\end{flalign*}
Let $u\in L^{p(x)}(\Omega)$ and $v\in L^{q(x)}(\Omega)$ with $\frac{1}{p(x)}+\frac{1}{q(x)}=1$, then we have the following H\"{o}lder type inequality
\begin{equation}\label{holder ineq}
\left|\int_{\Omega}uv\d x\right|\leq\left(\frac{1}{p^-}+\frac{1}{q^-}\right)\|u\|_{p(x)}\|v\|_{q(x)}.
\end{equation}
Let us now recall the modular function $\rho_{q(x)}:L^{q(x)}(\Omega)\rightarrow\mathbb{R}$ which plays an important role in the variable order Lebesgue spaces and which is defined by
$$\rho_{q(x)}(x)=\int_{\Omega}|u(x)|^{q(x)}\d x.$$
We now state the following Lemma due to \cite[Theorem 1.3, 1.4]{Fan2001}.
\begin{lemma}\label{lemma 1}
	Suppose, $u, u_n\in L^{q(x)}(\Omega)$ for all $n\in\mathbb{N}$. Then
	\begin{enumerate}[label=(\roman*)]
		\item $\|u\|_{q(x)}\lesseqqgtr1$ iff $\rho_{q(x)}(u)\lesseqqgtr1$.
		\item $\|u\|_{q(x)}<1\Rightarrow\|u\|_{q(x)}^{q^+}\leq\rho_{q(x)}(u)\leq\|u\|_{q(x)}^{q^-}$.
		\item $\|u\|_{q(x)}>1\Rightarrow\|u\|_{q(x)}^{q^-}\leq\rho_{q(x)}(u)\leq\|u\|_{q(x)}^{q^+}$.
		\item $\|u_n-u\|_{q(x)}\rightarrow0$ iff $\rho_{q(x)}(u_n-u)\rightarrow0$ iff $u_n\rightarrow u$ in measure, in $\Omega$ and $\rho_{q(x)}(u_n)\rightarrow\rho_{q(x)}(u)$.
	\end{enumerate}
\end{lemma}
\noindent For the development and properties concerning the Sobolev space $W^{k,p(x)}(\Omega)$ one may refer to \cite{Kovacik1991,Edmunds2000,Edmunds2002,repos2015} and the references therein.

\medskip

 For $0<s<1$, let us now define the variable order fractional Sobolev space as follows. Let $\Omega\subset\mathbb{R}^N$ be a smooth bounded domain and the assumptions \eqref{assm p}, \eqref{assm q} be true. Then the variable order fractional Sobolev space is defined by
\begin{align*}
W=&W^{s,q(x),p(x,y)}(\Omega)=\left\{u\in L^{q(x)}(\Omega): \iint_{\Omega\times\Omega}\cfrac{|u(x)-u(y)|^{p(x,y)}}{\mu^{p(x,y)}|x-y|^{N+sp(x,y)}}\d x\d y<\infty~\text{for some}~\mu>0 \right\}.
\end{align*}
The space $(W,\|\cdot\|_{W})$ is a Banach space endowed with the natural Gagliardo norm $\|u\|_{W}=\|u\|_{q(x)}+[u]_{s,p(x,y)}$, where $[u]_{s,p(x,y)}$ refers to the Gagliardo semi-norm and is given by
$$[u]_{s,p(x,y)}=\inf\left\{\mu>0: \iint_{\Omega\times\Omega}\cfrac{|u(x)-u(y)|^{p(x,y)}}{\mu^{p(x,y)}|x-y|^{N+sp(x,y)}}\d x\d y<1\right\}.$$
Moreover, the space $(W,\|\cdot\|_{W})$ is separable, reflexive Banach space \cite[Lemma 3.1]{Bahrouni2018}. Let $W_0$ denote the closure of $C_c^{\infty}(\Omega)$ in $W$, then $(W_0,\|\cdot\|_{W_0})$ is a Banach space equipped with the norm $$\|u\|_{W_0}=[u]_{s,p(x,y)}=\inf\left\{\mu>0: \iint_{\Omega\times\Omega}\cfrac{|u(x)-u(y)|^{p(x,y)}}{\mu^{p(x,y)}|x-y|^{N+sp(x,y)}}\d x\d y<1\right\}.$$
Similar to the Lemma \ref{lemma 1}, we have the following Lemma for the space $W_0$ with the modular function $\rho_{W_0}:W_0\rightarrow\mathbb{R}$ defined by 
$$\rho_{W_0}(u)=\iint_{\Omega\times\Omega}\cfrac{|u(x)-u(y)|^{p(x,y)}}{|x-y|^{N+sp(x,y)}}\d x\d y.$$
\begin{lemma}\label{lemma 2}
	Suppose, $u, u_n\in W_0$ for all $n\in\mathbb{N}$. Then
	\begin{enumerate}[label=(\roman*)]
		\item $\|u\|_{W_0}\lesseqqgtr1$ iff $\rho_{W_0}(u)\lesseqqgtr1$.
		\item $\|u\|_{W_0}<1\Rightarrow\|u\|_{W_0}^{p^+}\leq\rho_{W_0}(u)\leq\|u\|_{W_0}^{p^-}$.
		\item $\|u\|_{W_0}>1\Rightarrow\|u\|_{W_0}^{p^-}\leq\rho_{W_0}(u)\leq\|u\|_{W_0}^{p^+}$.
		\item $\|u_n-u\|_{W_0}\rightarrow0$ iff $\rho_{W_0}(u_n-u)\rightarrow0$ iff $u_n\rightarrow u$ in measure in $\Omega$ and $\rho_{W_0}(u_n)\rightarrow\rho_{W_0}(u)$.
	\end{enumerate}
\end{lemma}
\noindent We now state the following embedding result due to \citet[Theorem 1.1]{Kaufmann2017}.
\begin{theorem}\label{embed cpt}
	Let $\Omega\subset\mathbb{R}^N$ be a smooth bounded domain and $0<s<1$. Let $q(x)$ and $p(x,y)$ satisfies \eqref{assm q} and \eqref{assm p} respectively such that $N>sp(x,y)$ and $q(x)>p(x,x)$ for all $x,y\in\bar{\Omega}$. Assume that $r\in C(\bar{\Omega},(1,\infty))$ such that
	$$p_s^*(x)=\frac{Np(x,x)}{N-sp(x,x)}>r(x)\geq r^->1, ~\text{for}~x\in\bar{\Omega}.$$
Then there exists a constant $C=C(N,s,p,q,r,\Omega)$ such that for every $u\in W$ we have
$$\|u\|_{r(x)}\leq C\|u\|_W.$$
Therefore, the space $W=W^{s,q(x),p(x,y)}(\Omega)$ is continuously embedded into the variable order Lebesgue space $L^{r(x)}(\Omega)$ for all $r\in(1,p_s^*)$. Moreover, this embedding is compact. This embedding result holds even for the space $W_0$.
\end{theorem}
\begin{remark}
	Theorem \ref{embed cpt} holds true even for $q(x)=p(x,x)$. See \cite[Theorem 2.1]{Azroul2019} for the proof.
\end{remark}
\noindent Prior to stating our main result, we define a weak solution to the problem \eqref{prob1}.

\medskip 

\begin{definition}\label{weak soln}
	A function $u\in W_0$ is said to be a {\it weak} solution to \eqref{prob1}, if $\frac{\phi}{u^{\gamma(\cdot)}}\in L^1(\Omega)$ and
	\begin{equation}\label{weak main}
	\!\!\!\iint_{\Omega\times\Omega}\!\!\!\!\!\!\frac{|u(x)-u(y)|^{p(x,y)-2}(u(x)-u(y))(\phi(x)-\phi(y))}{|x-y|^{N+sp(x,y)}}\d x\d y-\int_{\Omega}\left(\frac{\lambda}{|u|^{\gamma(x)-1}u}+f(x,u)
	\right)\phi \d x=0,
	\end{equation}
	for all $\phi\in W_0.$
\end{definition}
\noindent The associated energy functional $I_{\lambda}:W_0\rightarrow(-\infty, \infty]$ for the problem \eqref{prob1} is given by
\begin{equation}\label{energy main}
I_{\lambda}(u):=\iint_{\Omega\times\Omega}\frac{|u(x)-u(y)|^{p(x,y)}}{p(x,y)|x-y|^{N+sp(x,y)}}\d x\d y-\int_{\Omega}\frac{\lambda}{1-\gamma(x)}|u|^{1-\gamma(x)}\d x-\int_{\Omega}F(x,u) \d x,
\end{equation}
where $F(x,t)=\int_{0}^{t}f(x,\tau)\d\tau$. With the above assumptions and definitions we conclude this section by stating the main result of the article.
\begin{theorem}\label{main thm}
	Suppose the assumptions \eqref{assm p},\eqref{assm q}, \eqref{assm gamma} and \eqref{assm f} are true. Then there exists $\Lambda<\infty$ such that for every $\lambda\in(0, \Lambda)$, the problem \eqref{prob1} has infinitely many nonnegative solutions $\{u_n\}\subset W_0\cap L^{\infty}(\bar{\Omega})$ with $I_{\lambda}(u_n)<0,$ $I_{\lambda}(u_n)\rightarrow0^-$ and $u_n\rightarrow0$ in $W_0.$
\end{theorem}
\section{Existence of infinitely many solutions}
\noindent This section is entirely devoted to obtain infinitely many solutions with the help of symmetric Mountain Pass Theorem due to Kajikiya \cite[Theorem 1(i)]{Kajikiya2005}. The main difficulty to apply the symmetric Mountain Pass Theorem is that the functional $I_{\lambda}$ fails to be $C^1(\Omega)$ due the presence of the singular term. Therefore, we first modify the problem \eqref{prob1} in the neighbourhood of $0$ by employing a cut-off technique developed in \cite{Clark1972}, which can also be found in \cite{Kajikiya2005}. Suppose the assumptions \eqref{assm f} holds true. Choose $0<l\leq\frac{1}{2} \min \{\delta, \delta'\}$, where $\delta$, $\delta'>0$ are as in \eqref{assm f}. Let us define a bounded $C^1$ function $\xi:\mathbb{R}\rightarrow\mathbb{R}^+$ with $0\leq\xi(t)\leq1$ such that $\xi(t)=1$, if $|t|\leq l$, $\xi(t)=0$ if $|t|\geq 2l$, $\xi$ is decreasing in $[l,2l]$ and increasing in $[-2l,-l]$ and define $\bar{f}(x, u)=f(x,u)\xi(u)$. Now consider the following cut-off problem
\begin{align}\label{prob2}
\begin{split}
(-\Delta)_{p(\cdot)}^{s}u&=\frac{\lambda}{|u|^{\gamma(x)-1}u}+\bar{f}(x,u)~\text{in}~\Omega,\\
u&=0~\text{in}~\mathbb{R}^N\setminus\Omega.
\end{split}
\end{align}
Then we say that a function $u\in W_0$ is a {\it weak} solution of \eqref{prob2}, if $\frac{\phi}{u^{\gamma(\cdot)}}\in L^1(\Omega)$ and
	\begin{equation}\label{weak cutoff}
	\iint_{\Omega\times\Omega}\frac{|u(x)-u(y)|^{p(x,y)-2}(u(x)-u(y))(\phi(x)-\phi(y))}{|x-y|^{N+sp(x,y)}}\d x\d y-\int_{\Omega}\left(\frac{\lambda}{|u|^{\gamma(x)-1}u}+\bar{f}(x,u)
	\right)\phi \d x=0,
	\end{equation}
	for all $\phi\in W_0.$
The associated energy functional $\bar{I}_{\lambda}:W_0\rightarrow(-\infty, \infty]$ corresponding to the problem \eqref{prob2} is given by
\begin{equation}\label{energy cutoff}
\bar{I}_{\lambda}(u)=\iint_{\Omega\times\Omega}\frac{|u(x)-u(y)|^{p(x,y)}}{p(x,y)|x-y|^{N+sp(x,y)}}\d x\d y-\int_{\Omega}\frac{\lambda}{1-\gamma(x)}|u|^{1-\gamma(x)}\d x-\int_{\Omega}\bar{F}(x, u)\d x,
\end{equation}
where $\bar{F}(x, t)=\int_0^t \bar{f}(x, \tau)d\tau$. It is easy to see that if $u$ is a solution to the problem \eqref{prob2} with $\|u\|_{\infty}\leq l$ then $u$ is also a solution to the problem \eqref{prob1}. Hence we will obtain infinitely many solutions to the problem \eqref{prob2} such that $\|u\|_{\infty}\leq l$. Prior to proving the existence results, let us recall the statement of symmetric Mountain Pass Theorem due to Kajikiya \cite{Kajikiya2005} followed by the definition of {\it genus} of a set.
\begin{definition}[\bf{Genus} \cite{Kajikiya2005}]\label{genus}
	Let $X$ be a Banach space and $A\subset X$. A set $A$ is said to be symmetric if $u\in A$ implies $-u\in A$. Let $A$ be a close, symmetric subset of $X$ such that $0\notin A$. We define a genus $\gamma(A)$ of $A$ by the smallest integer $k$ such that there exists an odd continuous mapping from $A$ to $\mathbb{R}^{k}\setminus\{0\}$. We define $\gamma(A)=\infty$, if no such $k$ exists.
\end{definition}
\begin{theorem}[Symmetric Mountain Pass Theorem \cite{Kajikiya2005}]\label{sym mountain}
	Let $X$ be an infinite dimensional Banach space and $I\in C^1(X,\mathbb{R})$ satisfies the following
	\begin{itemize}
		\item[(i)] $I$ is even, bounded below, $I(0)=0$ and $I$ satifies the $(PS)_c$ condition.
		\item[(ii)] For each $n\in\mathbb{N}$, there exists an $A_n\in\Gamma_n$ such that $\sup\limits_{u\in A_n}I(u)<0.$ 
	\end{itemize}
	Then for each $n\in\mathbb{N}$, $c_n=\inf\limits_{A\in \Gamma_n}\sup\limits_{u\in A}I(u)<0$ is a critical value of $I.$
\end{theorem}
\noindent We first state and prove the following Lemma which gives a finite range of $\lambda$. 
\begin{lemma}
	Suppose the assumptions \eqref{assm p},\eqref{assm q}, \eqref{assm gamma} and \eqref{assm f} hold true. Then $0\leq\Lambda<\infty$, where $\Lambda=\inf\{\lambda>0:~\text{The problem \eqref{prob1} has no solution}\}$.
\end{lemma}
\begin{proof}
It can be easily seen that $\Lambda\geq 0$. What we are left to check is whether $\Lambda<\infty$ or not. 
 This will be proved by contradiction. Let $\Lambda=\infty$.
Then it can be assumed that there exists $\lambda_n\rightarrow\infty$ such that $u_n$ is a solution to the problem \eqref{prob1}. On choosing
$\lambda_{*}>0$ such that
$$\lambda t^{-\gamma(x)}+t^{\alpha(x)}>(\lambda_1+\epsilon)t^{p(x)-1}, \forall~ t>0, \epsilon\in (0,1), \lambda>\lambda_{*}.$$
Corresponding to each choice of $\lambda_n>\lambda_{*}$ we have $u_n$, a solution of the problem \eqref{prob1}. Note that $\overline{u}=u_{\lambda}$ is a
supersolution to the following problem.
\begin{align}\label{auxen1}
\begin{split}
(-\Delta_p(x))^s u&=(\lambda_1+\epsilon)|u|^{p(x)-2}u~\text{in}~\Omega,\\
u&=0~\text{in}~\mathbb{R}^N\setminus\Omega.
\end{split}
\end{align}
Further, choose $r>0$ small enough such that $\underline{u}=r\phi_1$ is a subsolution to \eqref{auxen1}. Since,
$u_{\lambda}, \phi_1\in L^{\infty}(\Omega)$ then there exists $r>0$ sufficiently small such that $\underline{u}\leq\overline{u}$. Consider
the monotone iteration
 \begin{align}\label{auxen2}
 \begin{split}
 u_0&=r\phi_1\\
 (-\Delta_p(x))^s u_n&=(\lambda_1+\epsilon)u_{n-1}^{p(x)-1}~\text{and}~u_n>0~\text{in}~\Omega
 \end{split}
 \end{align}
Therefore, by using the weak comparison principle we have $r\phi_1\leq u_1(x)\leq u_2(x)\leq\cdots\leq u_n(x)\leq\cdots\leq u_{\lambda}(x)$, $\forall x\in\Omega$. Therefore the sequence $\{u_n\}$ is bounded in $X_0$ and hence $\{u_n\}$ has a weakly
convergent subsequence, still denoted by $\{u_n\}$, converging to $u_0$, which is a solution to the
problem \eqref{auxen1}. Since, $\epsilon>0$ is arbitrary and $\lambda_1$ is simple and isolated, hence
a contradiction is arrived at. Thus one can conclude that $\Lambda<\infty$.
\end{proof}
\noindent The next two consecutive Lemmas guarantees the validity of the hypothesis $(i)$ and $(ii)$ of the Theorem \ref{sym mountain} for the functional $\bar{I}_{\lambda}$.
\begin{lemma}\label{lemma ps}
	The functional $\bar{I}_{\lambda}$ is bounded below, even, $\bar{I}_{\lambda}(0)$=0 and satisfies $(PS)_c$ condition.
\end{lemma}
\begin{proof}
	Clearly, $\bar{I}_{\lambda}(0)=0$ and even. We now prove that the functional is bounded from below. Without loss of generality let us assume $\|u\|_{W_0}>1$. Then by virtue of the H\"{o}lder's inequality \eqref{holder ineq} together with the compact embedding, Theorem \ref{embed cpt} and by using the definition of $\xi$, we get
	\begin{align*}
	\bar{I}_{\lambda}(u)&\geq\frac{1}{p^+}\iint_{\Omega\times\Omega}\frac{|u(x)-u(y)|^{p(x,y)}}{|x-y|^{N+sp(x,y)}}\d x\d y-\frac{\lambda}{1-\gamma^+}\int_{\Omega}|u|^{1-\gamma(x)}\d x-C\\
	&\geq\frac{1}{p^+}\|u\|_{W_0}^{p^-}-\frac{C'\lambda}{1-\gamma^+}\|u\|_{W_0}^{1-\gamma^+}-C,
	\end{align*}
	where $C, C'$ are positive constants. This guarantees that $\bar{I}_{\lambda}$ is coercive as well as  bounded from below. Now suppose $\{u_n\}\subset W_0$ is a $(PS)_c$ sequence for $\bar{I}_{\lambda}$. Therefore, by using the coercivity of $\bar{I}_{\lambda}$, we have that the sequence $\{u_n\}$ is bounded and hence upto a subsequence $u_n\rightharpoonup u$ in $W_0$. Hence,  for every $\phi\in W_0$, we get
	\begin{align}\label{operator conv}
	\iint_{\Omega\times\Omega}&\frac{|u_n(x)-u_n(y)|^{p(x,y)-2}(u_n(x)-u_n(y))(\phi(x)-\phi(y))}{|x-y|^{N+sp(x,y)}}\d x\d y\nonumber\\
	&\longrightarrow \iint_{\Omega\times\Omega}\frac{|u(x)-u(y)|^{p(x,y)-2}(u(x)-u(y))(\phi(x)-\phi(y))}{|x-y|^{N+sp(x,y)}}\d x\d y.
	\end{align}
	Moreover, by the Theorem \ref{embed cpt}, we have as $n\rightarrow\infty$
	\begin{align}\label{embed strong}
	u_n&\longrightarrow u ~\text{in}~ L^{r(x)}(\Omega),~\forall~1<r(x)< p_s^*(x)~\text{and}~
	u_n(x)\longrightarrow u(x) ~\text{a.e.}~ \Omega.
	\end{align}
	{\bf{Claim}:} Under the assumption \eqref{assm f}, we have
	\begin{equation}\label{f conv}
	\int_{\Omega}\bar{f}(x,u_n)u\d x\rightarrow\int_{\Omega}\bar{f}(x,u)u\d x ~\text{and}~ \int_{\Omega}\bar{f}(x,u_n)u_n\d x\rightarrow\int_{\Omega}\bar{f}(x,u)u\d x.
	\end{equation}
	By using \eqref{embed strong} and Lemma \ref{appendix a1 gen}, there exists a $g\in L^{r(x)}(\Omega)$ such that
	\begin{equation}\label{appendeix A1}
	|u_n(x)|\leq g(x) ~\text{a.e. in}~ \Omega, \forall\,n\in\mathbb{N}.
	\end{equation}
	Therefore, from \eqref{embed strong}, \eqref{appendeix A1} and by using Lebesgue dominated convergence theorem and a Brezis-Lieb type Lemma \cite{Brezis1983}, our claim is established.\\
	Now on using the H\"{o}lder's inequality \eqref{holder ineq} with $r_1(x)$ and $r_1'(x)$ such that $r_1(x)(1-\gamma(x))>1$ and then on passing the limit $n\rightarrow\infty$, we have
	\begin{align}\label{gamma est1}
	\begin{split}
	\int_{\Omega}|u_n|^{1-\gamma(x)}\d x&\leq\int_{\Omega}|u|^{1-\gamma(x)}\d x+\int_{\Omega}|u_n-u|^{1-\gamma(x)}\d x\\
	&\leq\int_{\Omega}|u|^{1-\gamma(x)}\d x+C\|u_n-u\|_{L^{r_1(x)}(\Omega)}^{1-\gamma^+}\\
	&=\int_{\Omega}|u|^{1-\gamma(x)}\d x +o(1)
	\end{split}
	\end{align}
	On proceeding in similar way, we obtain
	\begin{align}\label{gamma est2}
	\begin{split}
	\int_{\Omega}|u|^{1-\gamma}dx&\leq\int_{\Omega}|u_n|^{1-\gamma}\d x +o(1)
	\end{split}
	\end{align}
	Therefore, from \eqref{gamma est1} and \eqref{gamma est2} it follows that
	\begin{equation}\label{convergence gamma}
	\int_{\Omega}u_n^{1-\gamma}dx=\int_{\Omega}u^{1-\gamma}\d x+o(1)
	\end{equation}
	Now by using the fact that $\langle\bar{I}_{\lambda}'(u_n),u_n\rangle\rightarrow0$ as $n\rightarrow\infty$, we get
	\begin{equation}\label{convergence I tilla}
	\iint_{\Omega\times\Omega}\frac{|u_n(x)-u_n(y)|^{p(x,y)}}{|x-y|^{N+sp(x,y)}}\d x\d y-\lambda\int_{\Omega}|u_n|^{1-\gamma(x)}\d x-\int_{\Omega}\bar{f}(x, u_n)u_n\d x\rightarrow0
	\end{equation}
	Hence the estimates \eqref{f conv}, \eqref{convergence gamma} and \eqref{convergence I tilla} gives
	\begin{equation}\label{convergence u_n}
	\iint_{\Omega\times\Omega}\frac{|u(x)-u(y)|^{p(x,y)}}{|x-y|^{N+sp(x,y)}}\d x\d y\rightarrow\lambda\int_{\Omega}|u|^{1-\gamma(x)}\d x+\int_{\Omega}\bar{f}(x, u)u\d x
	\end{equation}
	Again, we have $\langle\bar{I}_{\lambda}'(u_n),u\rangle\longrightarrow0$ as $n\rightarrow\infty$ and 
	\begin{align}
	\begin{split}
	\langle\bar{I}_{\lambda}'(u_n),u\rangle&=\iint_{\Omega\times\Omega}\frac{|u_n(x)-u_n(y)|^{p(x,y)-2}(u_n(x)-u_n(y))(u(x)-u(y))}{|x-y|^{N+sp(x,y)}}\d x\d y\\
	&\hspace{2cm}-\lambda\int_{\Omega}\frac{u}{|u_n|^{\gamma(x)-1}u_n}dx-\int_{\Omega}\bar{f}(x, u_n)udx
	\end{split}
	\end{align}
	Therefore, on choosing $\phi=u$ as the test function in \eqref{operator conv} and then by using \eqref{convergence gamma}-\eqref{convergence u_n}, we obtain
	\begin{equation}\label{convergence norm u}
	\iint_{\Omega\times\Omega}\frac{|u(x)-u(y)|^{p(x,y)}}{|x-y|^{N+sp(x,y)}}\d x\d y=\lambda\int_{\Omega}|u|^{1-\gamma(x)}\d x+\int_{\Omega}\bar{f}(x, u)u\d x
	\end{equation}
	Hence, $\|u_n\|_{W_0}\rightarrow\|u\|_{W_0}$ which guarantees that the functional $\bar{I}_{\lambda}$ satisfies $(PS)_c$ condition.
\end{proof}
\noindent Define $\Gamma_n=\{A_n\subset W_0: A_n~\text{is closed, symmetric and}~ 0\notin A_n~\text{such that the genus}~ \gamma(A_n)\geq n\}.$
\begin{lemma}\label{lemma genus}
	For any $n\in\mathbb{N}$, there exists a closed, symmetric subset $A_n\subset W_0$ with $0\notin A_n$ such that $\gamma(A_n)\geq n$ and $\sup_{u\in A_n}\bar{I}_{\lambda}(u)<0.$
\end{lemma}
\begin{proof}
	We first construct a finite dimensional subspace of $W_0$. Observe that by using the definition of $\xi$, there exists $R>0$ such that $\bar{f}(x,t)\leq R$ and $\bar{F}(x,t)\leq R$ for all $(x,t)\in\Omega\times\mathbb{R}$. Now since $\alpha^+<p^-$, then by using the definition of $\bar{f}$, we get
	\begin{equation}\label{symm set 1}
	\bar{F}(x,u)=\xi(u)f(x,u)\geq\frac{C|u|^{\alpha^+}}{\alpha^+}~\forall\,x\in\Omega~\text{and}~|u|\leq l,
	\end{equation}
	for some $C>0$ and $0<l\leq 1$. Now for a fix $k\in\mathbb{N}$, choose, $\{\phi_1, \phi_2,\cdots,\phi_k\}\subset C_0^{\infty}(\Omega)$ such that $\phi_i\neq0$, $supp(\phi_i)\subset\Omega$ for all $i=1,2,\cdots,k$ and $supp(\phi_i)\cap supp(\phi_j)=\emptyset$ for $i\neq j$. Define,
	$W_k=span\{\phi_1, \phi_2,\cdots,\phi_k\}.$ Clearly $W_k$ is a finite subspace of $W_0$. Now, since all norms on a finite dimensional subspace are equivalent then there exist $c_k, d_k>0$ such that
	\begin{equation}\label{symm set 2}
	\|u\|_{W_0}\geq c_k\|u\|_{\infty} ~\text{and}~ \|u\|_{\alpha^+}\leq d_k\|u\|_{W_0}.
	\end{equation} 
	Finally, on choosing $\rho_k=\left\{\frac{l}{2},\frac{lc_k}{2},\left(\frac{Clp^-d_k^{\alpha^+}}{2\alpha^+}\right)^{\frac{1}{p^--\alpha^+}}\right\}$ and by using Lemma \ref{lemma 2}, \eqref{symm set 1}, \eqref{symm set 2}, we get
	\begin{align*}
	\bar{I}_{\lambda}(u)&\leq\frac{1}{p^-}\iint_{\Omega\times\Omega}\frac{|u(x)-u(y)|^{p(x,y)}}{|x-y|^{N+sp(x,y)}}\d x\d y-\frac{\lambda}{1-\gamma^-}\int_{\Omega}|u|^{1-\gamma(x)}\d x-\int_{\Omega}\bar{F}(x, u)\d x\\
	&\leq\frac{1}{p^-}\|u\|_{W_0}^{p^-}-\frac{\lambda}{1-\gamma^-}\int_{\Omega}|u|^{1-\gamma(x)}\d x-\frac{C}{\alpha^+}\int_{\Omega}|u|^{\alpha(x)}\d x\\
	&\leq\left[\frac{1}{p^-}\|u\|_{W_0}^{p^-}-\frac{Cd_k^{\alpha^+}}{\alpha^+}\|u\|_{W_0}^{\alpha^+}\d x\right]-\frac{\lambda}{1-\gamma^-}\int_{\Omega}|u|^{1-\gamma(x)}\d x\\
	&<0,
	\end{align*}
	for all $u\in W_k\cap S_{\rho_k}$, where $S_{\rho}=\{u\in W_k:\|u\|_{W_0}=\rho\}$. Set, $A_n:=\{u\in W_n: \|u\|_{W_0}=\rho_n\}$. Then $\Gamma_n\neq\phi$. Moreover, $A_n$ is symmetric, closed such that $\gamma(A_n)\geq n$ and $\sup_{u\in A_n}\bar{I}_{\lambda}(u)<0$. This completes the proof.
\end{proof}
\begin{proof}[{\bf Proof of Theorem \ref{main thm}}]
	We can see that Lemma \ref{lemma ps} and Lemma \ref{lemma genus} satisfies the hypothesis $(i)$ and $(ii)$ of Theorem \ref{sym mountain}. Therefore, we conclude that the functional $\bar{I}_{\lambda}$ has sequence of critical points $\{u_n\}$ such that $\bar{I}_{\lambda}(u_n)<0$ and $\bar{I}_{\lambda}(u_n)\rightarrow0^-$. We first prove the nonnegativity of solutions to the problem \eqref{prob1}. Consider $\Omega= \Omega^+\cup\Omega^-$, where $\Omega^+=\{x\in W_0: u_n(x)\geq0 \}$ and $\Omega^-=\{x\in W_0: u_n(x)<0 \}$. Also, let us define $u_n(x)=u_n^+-u_n^-$, where $u_n^+(x)=\max\{u_n(x), 0\}$ and $u_n^-(x)=\max\{-u_n(x), 0\}$. We will proceed by the method of contradiction. Suppose, $u_n<0$ a.e. in $\Omega$. On choosing, $\phi=u_n^-$ as the test function in \eqref{weak main} and using $(a-b)(a^--b^-)\leq-(a^--b^-)^2$, we have
	\begin{align*}
	&\int_{\Omega}\!\left(\lambda\frac{u_n^-}{|u_n|^{\gamma(x)-1}u_n}+f(x,u_n)u_n^-\right)\d x\!=\!\!\!\iint\limits_{\Omega\times\Omega}\!\!\frac{|u_n(x)-u_n(y)|^{p(x,y)-2}(u_n(x)-u_n(y))(u_n^-(x)-u_n^-(y))}{|x-y|^{N+2s}}\d x\d y\\
	&\Rightarrow\lambda\int_{\Omega}\frac{sign(u_n)u_n^-}{|u_n|^{\gamma(x)}}\d x\leq-\iint\limits_{\Omega\times\Omega}\frac{|u_n(x)-u_n(y)|^{p(x,y)-2}(u_n^-(x)-u_n^-(y))^2}{|x-y|^{N+2s}}\d x\d y\\
	&\Rightarrow\lambda\int_{\Omega^-}|u_n^-|^{1-\gamma(x)}dx\leq -\rho_{W_0}(u_n^-)<0.
	\end{align*}
	Therefore, by using Lemma \ref{lemma 2}, we can conclude that $|\Omega^-|=0$. This contradicts the assumption $u_n<0$ a.e. in $\Omega$ and hence the solutions to \eqref{prob1} are nonnegative. Now observe that $o_n(1)=\frac{1}{\alpha^{+}}\langle{\bar{I}_{\lambda}^{'}(u_n), u_n\rangle-\bar{I}_{\lambda}(u_n)}$. Moreover, from the definition of $\bar{I}_{\lambda}$, we have
	\begin{align*}
	\frac{1}{\alpha^+}\langle\tilde{I}^{'}(u_n), u_n\rangle-\tilde{I}(u_n)&=\frac{1}{\alpha^+}\left[\iint_{\Omega\times\Omega}\frac{|u_n(x)-u_n(y)|^{p(x,y)}}{|x-y|^{N+sp(x,y)}}\d x\d y-\int_{\Omega}\left(\lambda|u_n|^{1- \gamma(x)} +\bar{f}(x,u_n)u_n\right)\d x\right]\\
	&\hspace{0.5cm}-\left[\iint_{\Omega\times\Omega}\frac{|u_n(x)-u_n(y)|^{p(x,y)}}{p(x,y)|x-y|^{N+sp(x,y)}}\d x\d y-\int_{\Omega}\left(\frac{\lambda|u_n|^{1-\gamma(x)}}{1- \gamma(x)}+\bar{F}(x,u_n)\right)\d x\right]\\
	&\geq(\frac{1}{\alpha^+}-\frac{1}{p^-})\iint_{\Omega\times\Omega}\frac{|u_n(x)-u_n(y)|^{p(x,y)}}{|x-y|^{N+sp(x,y)}}\d x\d y\\
	&\hspace{0.8cm}-\lambda(\frac{1}{\alpha^+}-\frac{1}{1-\gamma^-})\int_{\Omega}|u_n|^{1-\gamma(x)}\d x+\frac{1}{\alpha^+}\int_{\Omega}(\alpha^+\bar{F}(x,u_n)-\bar{f}(x,u_n))\d x\\
	&\geq(\frac{1}{\alpha^+}-\frac{1}{p^-})\rho_{W_0}(u_n)+\lambda(\frac{1}{1-\gamma^-}-\frac{1}{\alpha^+})\int_{\Omega}|u_n|^{1-\gamma(x)}\d x\\
	&\geq(\frac{1}{\alpha^+}-\frac{1}{p^-})\rho_{W_0}(u_n).
	\end{align*}
	Therefore, $o_n(1)=\frac{1}{\alpha^{+}}\langle{\bar{I}_{\lambda}^{'}(u_n), u_n\rangle-\bar{I}_{\lambda}(u_n)}\geq(\frac{1}{\alpha^+}-\frac{1}{p^-})\rho_{W_0}(u_n)$. Hence, we get $\rho_{W_0}(u_n)\rightarrow0$ and this implies that $u_n\rightarrow0$ in $W_0.$ Thus, we can conclude that the problem \eqref{prob2} has infinitely many solutions. Finally, thanks to the Moser iteration, from Lemma \ref{bounded}, we can obtain $\|u_n\|_{L^{\infty}(\Omega)}\leq l$ as $n\rightarrow\infty$. Therefore, the problem \eqref{prob1} has infinitely many nonnegative weak solutions.
\end{proof}
\section{Boundedness of solutions to \eqref{prob1}.}
\noindent In this section we obtain a uniform bound of weak solutions to the problem \eqref{prob2} by using Moser iteration method followed by some auxilliary Lemmas. Let us begin with the following Lemma.
\begin{lemma}\label{beta convex}
	For every $\beta(x)>0$ and $p(x)\geq 1$ we have
	$$\left(\frac{1}{\beta(x)}\right)^{\frac{1}{p(x)}}\left(\frac{p(x)+\beta(x)-1}{p(x)}\right)\geq 1.$$
\end{lemma}
\begin{proof}
	Observe that for $p(x)=1$, the result is obvious. Therefore, we may assume $p(x)>1$. Now for $p(x)>1$ the function $t\mapsto t^{p(x)}$ is convex. Hence,
	$$\beta(x)-1\geq p(x)(\beta(x)^{\frac{1}{p(x)}}-1).$$
	Now on adding $p(x)$ to both sides we obtain
	$$p(x)+\beta(x)-1\geq p(x)\beta(x)^{\frac{1}{p(x)}}.$$
	Hence, we get
	$$\left(\frac{1}{\beta(x)}\right)^{\frac{1}{p(x)}}\left(\frac{p(x)+\beta(x)-1}{p(x)}\right)\geq 1.$$
\end{proof}
\noindent We now consider the monotone increasing function $J_{p(x)}(t):=|t|^{p(x)-2}t$ for every $1<p(x)<\infty$.
\begin{lemma}\label{l infty 1}
	Assume $1<p(x)<\infty$ and $f: \mathbb{R}\rightarrow \mathbb{R}$ to be a $C^{1}$ convex function. Then for any $\tau\geq 0$
	\begin{equation}\label{bdd est1}
	J_{p(x)}(a-b)\big[AJ_{p(x),\tau}(f'(a))-BJ_{p(x),\tau}(f'(b))\big]\geq(\tau(a-b)^{2}+(f(a)-f(b))^{2})^{\frac{p(x)-2}{2}}(f(a)-f(b))(A-B),
	\end{equation}
	for every $a, b\in \mathbb{R}$ and every $A, B\geq 0$, where $J_{p(x),\tau}(t)=(\tau+|t|^{2})^{\frac{p-2}{2}}t,~ t\in \mathbb{R}.$
\end{lemma}
\begin{proof}
	The result is trivial if $a=b$. Therefore, let us assume $a\neq b$. Since the function $f$ is $C^1$ and convex then
	\begin{equation}\label{convex f}
	f(a)-f(b)\leq f'(a)(a-b)~\text{and}~ f(a)-f(b)\geq f'(b)(a-b)
	\end{equation}
	Now from the left hand side of \eqref{bdd est1}, we get
	\begin{align*}
	J_{p(x)}(a-b)&\left[AJ_{p(x),\tau}(f'(a))-BJ_{p(x),\tau}(f'(b))\right]\\
	&=\left(\tau(a-b)^{2}+(f'(a)(a-b))^{2}\right)^{\frac{p(x)-2}{2}}f'(a)(a-b)A\\
	&-\left(\tau(a-b)^{2}+(f'(b)(a-b))^{2}\right)^{\frac{p(x)-2}{2}}f'(b)(a-b)B,
	\end{align*}
	Observe that the function $F(t)=\frac{1}{p(x)}\left(\tau(a-b)^{2}+t^{2}\right)^{\frac{p(x)}{2}}$ is convex and its derivative is given by $F'(t)=\left(\tau(a-b)^{2}+t^{2}\right)^{\frac{p(x)-2}{2}}t$. Then by simplifying \eqref{convex f} and by the monotonicity of $F'$, we can obtain the desired inequality \eqref{bdd est1}.
\end{proof}
\begin{remark}
	For $\tau=0$, the estimate of the above Lemma becomes
	\begin{equation}\label{bdd est1 remark}
	J_{p(x)}(a-b)\big[AJ_{p(x)}(f'(a))-BJ_{p(x)}(f'(b))\big]\geq(f(a)-f(b))^{p(x)-2}(f(a)-f(b))(A-B),
	\end{equation}
	for every $a, b\in \mathbb{R}$ and every $A, B\geq 0$
\end{remark}
\begin{lemma}\label{l infty 2}
	Assume $1<p(x)<\infty$ and $g:\mathbb{R}\rightarrow \mathbb{R}$ to be an increasing function. Define
	$$G(t)=\int_{0}^{t}g'(\tau)^{\frac{1}{p(x)}}\d\tau, t\in \mathbb{R},$$
	then we have
	\begin{equation}\label{bdd est2}
	J_{p(x)}(a-b)(g(a)-g(b))\geq|G(a)-G(b)|^{p(x)}.
	\end{equation}
\end{lemma}
\begin{proof}
	Suppose $a>b$ without the loss of generality. Then we have
	\begin{align*}
	J_{p(x)}(a-b)(g(a)-g(b))&=(a-b)^{p(x)-1}\int_{b}^{a}g'(\tau)\d\tau\\
	&=(a-b)^{p(x)-1}\int_{b}^{a}G'(\tau)^{p(x)}\d\tau\\
	&\geq\left(\int_{b}^{a}G'(\tau)\d\tau\right)^{p(x)}~\text{(by using Jensen inequality)}\\
	&=|G(a)-G(b)|^{p(x)}.
	\end{align*}
\end{proof}
\begin{lemma}\label{brasco est extra}
	Suppose	$\beta(x)\geq 1$, then for every $a, b\geq 0$ we have
	\begin{equation}\label{bdd est3}
	|a-b|^{p(x)}(a^{\beta(x)-1}+b^{\beta(x)-1})\leq(\max\{1,(3-\beta(x))\})|a-b|^{p(x)-2}(a-b)(a^{\beta(x)}-b^{\beta(x)}).
	\end{equation}
\end{lemma}
\begin{proof}
	Observe that the estimate \eqref{bdd est3} is true for $a=b$. Therefore, one can assume $a>b$. Then the estimate \eqref{bdd est3} reduces to
	$$(1-t)^{p(x)}(1+t^{\beta(x)-1})\leq C(1-t)^{p(x)-1}(1-t^{\beta(x)})~\text{for}~0\leq t<1,$$
	which implies that
	$$(1-t)(1+t^{\beta(x)-1})\leq C(1-t)(1-t^{\beta(x)})~\text{for}~0\leq t<1.$$
	Note that $(1-t)(1+t^{\beta(x)-1})=(1-t^{\beta(x)})+t^{\beta(x)-1}-t$. Since $0\leq t<1$, then for every $\beta(x)\geq2$ we have $t^{\beta(x)-1}-t\leq 0$. On the other hand for $1<\beta(x)<2$, the function $\tau\mapsto\tau^{\beta(x)-1}$ is concave. Therefore, on using the concavity we get
	$$t^{\beta(x)-1}-t=(t^{\beta(x)-1}-1)-(t-1)\leq(\beta(x)-1)(t-1)-(t-1)\leq(2-\beta(x))(1-t^{\beta(x)}).$$
	The estimate \eqref{bdd est3} holds trivially for $\beta(x)=1$. Hence we can conclude the proof for every $\beta(x)\geq1$. 
\end{proof}

\begin{lemma}\label{bounded}
	Suppose the assumptions \eqref{assm p},\eqref{assm q}, \eqref{assm gamma} and \eqref{assm f} are fulfilled. Let $u\in W_0$ be a solution to the problem \eqref{prob2}. Then $u\in L^{\infty}(\bar{\Omega}).$
\end{lemma}
\begin{proof}
	Let us first define the smooth function $g_{\epsilon}(t)=(\epsilon^2+t^2)^{\frac{1}{2}}$ for every $\epsilon>0.$ Note that the function $g_{\epsilon}$ is convex and Lipschitz. For each $0<\psi\in C_c^{\infty}(\Omega)$, we choose $\phi=\psi |g'_{\epsilon}(u)|^{p(x,y)-2}g'_{\epsilon}(u)$ as the test function in \eqref{weak cutoff}. By choosing $a=u(x), b=u(y), A=\psi(x)$ and $B=\psi(y)$ in Lemma \ref{l infty 1} we have
	\begin{align}\label{bound est 2.0}
	\iint_{\Omega\times\Omega}&\cfrac{|g_{\epsilon}(u(x))-g_{\epsilon}(u(y))|^{p(x,y)-2}(g_{\epsilon}(u(x))-g_{\epsilon}(u(y)))(\psi(x)-\psi(y))}{|x-y|^{N+sp(x,y)}}\d x\d y\nonumber\\
	&\leq\int_\Omega\left(\left|\frac{\lambda}{|u|^{\gamma(x)-1}u}+\bar{f}(x, u)\right|\right)|g_{\epsilon}(u)|^{p(x)-1}\psi \d x
	\end{align}
	for all $0<\psi\in C_c^{\infty}(\Omega)$. Moreover, $g_{\epsilon}(t)\rightarrow|t|$ as $t\rightarrow0$ and $|g'_{\epsilon}(t)|\leq1.$ Therefore, on using Fatou's Lemma and passing the limit $\epsilon\rightarrow0$, we obtain
	\begin{align}\label{bound est 2}
	\!\!\!\!\!\int_{Q}\!\!\cfrac{||u(x)|-|u(y)||^{p(x,y)-2}(|u(x)|-|u(y)|)(\psi(x)-\psi(y))}{|x-y|^{N+sp(x,y)}}\d x\d y\leq\!\!\int_\Omega\!\!\left(\left|\frac{\lambda}{|u|^{\gamma(x)-1}u}+\bar{f}(x, u)\right|\right)\psi \d x
	\end{align}
	The inequality \eqref{bound est 2} holds true even for every $\psi\in X_0$. Now define $u_k=\min\{(u-1)^+, k\}\in X_0$ for each $k>0$. For any given $\beta>0$ and $\delta>0$ choose $\psi=(u_k+\delta)^{\beta}-\delta^{\beta}$ as the test function in \eqref{bound est 2}. This gives
	\begin{align*}
	\int_{Q}&\cfrac{||u(x)|-|u(y)||^{p(x,y)-2}(|u(x)|-|u(y)|)((u_k(x)+\delta)^{\beta}-(u_k(y)+\delta)^{\beta})}{|x-y|^{N+sp(x,y)}}\d x\d y\\
	&\leq\int_\Omega\left|\frac{\lambda}{|u|^{\gamma(x)-1}u}+\bar{f}(x, u)\right|((u_k+\delta)^{\beta}-\delta^{\beta}) \d x
	\end{align*}
	Now on applying Lemma \ref{l infty 2} to the function $h(u)=(u_k+\delta)^{\beta}$ we obtain
	\begin{align}\label{bound est 4}
	\begin{split}
	&\int_{Q}\cfrac{|((u_k(x)+\delta)^{\frac{\beta+p(x,y)-1}{p(x,y)}}
		-(u_k(y)+\delta)^{\frac{\beta+p(x,y)-1}{p(x,y)}})|^{p(x,y)}}{|x-y|^{N+sp(x,y)}}\d x\d y\\
	&\!\leq\int_{Q}\!\!\left(\cfrac{(\beta+p(x,y)-1)^{p(x,y)}}{{\beta}p(x,y)^{p(x,y)}}\right)\cfrac{||u(x)|-|u(y)||^{p(x,y)-2}(|u(x)|-|u(y)|)((u_k(x)+\delta)^{\beta}-(u_k(y)+\delta)^{\beta})}{|x-y|^{N+sp(x,y)}}\d x\d y\\
	&\leq\left(\cfrac{(\beta+p^+-1)^{p^+}}{\beta(p^+)^{p^+}}\right)\int_{Q}\cfrac{||u(x)|-|u(y)||^{p(x,y)-2}(|u(x)|-|u(y)|)((u_k(x)+\delta)^{\beta}-(u_k(y)+\delta)^{\beta})}{|x-y|^{N+sp(x,y)}}\d x\d y\\
	&\leq\left(\cfrac{(\beta+p^+-1)^{p^+}}{\beta(p^+)^{p^+}}\right)\int_{\Omega}\left(\left|\frac{\lambda}{|u|^{\gamma(x)-1}u}\right|+|\bar{f}(x,u)|\right)\left((u_k+\delta)^{\beta}-\delta^{\beta}\right) \d x\\	
	&\leq{C_1}\left(\cfrac{(\beta+p^+-1)^{p^+}}{\beta(p^+)^{p^+}}\right)\left[\int_{\Omega}\lambda|u|^{-\gamma(x)}\left((u_k+\delta)^{\beta}-\delta^{\beta}\right)+\int_{\Omega}|u|^{\alpha(x)}\left((u_k+\delta)^{\beta}-\delta^{\beta}\right) \d x\right]\\
	&={C_1}\left(\cfrac{(\beta+p^+-1)^{p^+}}{\beta(p^+)^{p^+}}\right)\left[\int_{\{u\geq1\}}\lambda|u|^{-\gamma(x)}\left((u_k+\delta)^{\beta}-\delta^{\beta}\right)+\int_{\{u\geq1\}}|u|^{\alpha(x)}\left((u_k+\delta)^{\beta}-\delta^{\beta}\right) \d x\right]\\
	&\leq{CC_1}\left(\cfrac{(\beta+p^+-1)^{p^+}}{\beta(p^+)^{p^+}}\right)\left[\int_{\{u\geq1\}}\left(1+|u|^{\alpha(x)}\right)\left((u_k+\delta)^{\beta}-\delta^{\beta}\right) \d x\right]\\
	&\leq{C'}\left(\cfrac{(\beta+p^+-1)^{p^+}}{\beta(p^+)^{p^+}}\right)\left[\int_{\{u\geq1\}}|u|^{\alpha(x)}\left((u_k+\delta)^{\beta}-\delta^{\beta}\right) \d x\right]\\
	&\leq{C'}\left(\cfrac{(\beta+p^+-1)^{p^+}}{\beta(p^+)^{p^+}}\right)\left[\int_{\Omega}|u|^{\alpha^+}\left((u_k+\delta)^{\beta}-\delta^{\beta}\right) \d x\right] \leq{C'}\left(\cfrac{(\beta+p^+-1)^{p^+}}{\beta(p^+)^{p^+}}\right)\|u\|_{r^*}^{\alpha^+}\|(u_k+\delta)^{\beta}\|_{t},
	\end{split}
	\end{align}
	where $t=\frac{r^*}{r^*-\alpha^+}$, $(t^*)<(p_s^*)^-$ and $C=\max\{1,|\lambda|\}.$ We now impose the Sobolev embedding theorem from \cite[Theorem 1.1]{Bahrouni2018a} to obtain
	\begin{align}\label{bound est 5}
	&\int_{Q}\cfrac{|((u_k(x)+\delta)^{\frac{\beta+p(x,y)-1}{p(x,y)}}
		-(u_k(y)+\delta)^{\frac{\beta+p(x,y)-1}{p(x,y)}})|^{p(x,y)}}{|x-y|^{N+sp(x,y)}}\d x\d y\geq{C}\left\|(u_k+\delta)^{\frac{\beta+p(x,y)-1}{p(x,y)}}-\delta^{\frac{\beta+p(x,y)-1}{p(x,y)}}\right\|_{r_{}^*}^{p(x,y)}
	\end{align}
	where $C>0$. The triangle inequality and $(u_k+\delta)^{\beta+p(x,y)-1}\geq\delta^{p(x,y)-1}(u_k+\delta)^{\beta}$ implies
	\begin{align}\label{bound est 6}
	\left[\int_{\Omega}\left((u_k+\delta)^{\frac{\beta+p(x,y)-1}{p(x,y)}}
	-\delta^{\frac{\beta+p(x,y)-1}{p(x,y)}}\right)^{r^*}\d x\right]^{\cfrac{p(x,y)}{r^*}}\geq\left(\frac{\delta}{2}\right)^{p(x,y)-1}&\left[\int_{\Omega}(u_k+\delta)^{\frac{r^*\beta}{p(x,y)}}\right]^{\cfrac{p(x,y)}{r^*}}\nonumber\\
	&-\delta^{\beta+p(x,y)-1}|\Omega|^{\cfrac{p(x,y)}{r^*}}.
	\end{align}
	Therefore, by using \eqref{bound est 6} in \eqref{bound est 5} and then from \eqref{bound est 4}, we have
	\begin{align}\label{bdd1}
	\!\!\!\!\!\!\left\|(u_k+\delta)^{\frac{\beta}{p(x,y)}}\right\|^{p(x,y)}_{r^*}
	\!\!\!\!\leq{C'}\!\!\left[C\left(\frac{2}{\delta}\right)^{p(x,y)-1}\!\!\left(\cfrac{(\beta+p^+-1)^{p^+}}{\beta(p^+)^{p^+}}\right)\|u\|_{r^*}^{\alpha^+}\|(u_k+\delta)^{\beta}\|_{t}+\delta^{\beta}|\Omega|^{\cfrac{p(x,y)}{r^*}}\right].
	\end{align}
	Now by using Lemma \ref{beta convex}, one can derive that
	\begin{align}\label{bdd2}
	\delta^{\beta}|\Omega|^{\cfrac{p(x,y)}{r^*}}&\leq\left(\cfrac{(\beta+p^+-1)^{p^+}}{\beta(p^+)^{p^+}}\right)|\Omega|^{\cfrac{p(x,y)}{r^*}-\cfrac{1}{t}}\left\|(u_k+\delta)^{\beta}\right\|_{t}\nonumber\\
	&\leq\left(\cfrac{(\beta+p^+-1)^{p^+}}{\beta(p^+)^{p^+}}\right)|\Omega|^{\cfrac{p^+}{r^*}-\cfrac{1}{t}}\left\|(u_k+\delta)^{\beta}\right\|_{t}
	\end{align}
	Therefore, on using \eqref{bdd2} in \eqref{bdd1}, we can deduce that
	\begin{align}\label{bound est 7}
	\left\|(u_k+\delta)^{\frac{\beta}{p^+}}\right\|^{p^+}_{r^*}
	&\leq{C'}\left[\frac{1}{\beta}\left(\cfrac{\beta+p^+-1}{p^+}\right)^{p^+}\left\|(u_k+\delta)^{\beta}\right\|_{t}\left(\frac{C\|u\|_{r^*}^{\alpha^+}}{\delta^{p-1}}+|\Omega|^{\cfrac{p^+}{r^*}-\cfrac{1}{t}} \right)\right].
	\end{align}	
	\noindent Now choose, $\delta>0$ such that $\delta^{p-1}=C\|u\|_{r^*}^{\alpha^+}\left(|\Omega|^{\frac{p^+}{r^*}-\frac{1}{t}}\right)^{-1}$ and $\beta\geq1$ with $\left(\frac{\beta+p^+-1}{p^+}\right)^{p^+}\leq\beta^{p^+}.$ Further, by setting $\eta=\cfrac{r^*}{tp^+}>1$ (such a choice for $t$ and $r^*$ is possible) and $\tau=t\beta$ we can rewrite the inequality \eqref{bound est 7} as
	\begin{align}\label{bound est 8}
	\left\|(u_k+\delta)\right\|_{\eta\tau}\leq\left(C|\Omega|^{\frac{p^+}{r^*}-\frac{1}{t}}\right)^{\frac{t}{\tau}}\left(\frac{\tau}{t}\right)^{\frac{t}{\tau}}\left\|(u_k+\delta)\right\|_{\tau}.
	\end{align}
	Set $\tau_0=t ~~\text{and}~~ \tau_{m+1}=\eta\tau_m=\eta^{m+1}t$. Then after performing $m$ iterations, the inequality \eqref{bound est 8} reduces to
	\begin{align}\label{bound est 9}
	\left\|(u_k+\delta)\right\|_{\tau_{m+1}}\leq\left(C|\Omega|^{\frac{p^+}{r^*}-\frac{1}{t}}\right)^{\left(\sum\limits_{i=0}^{m}\frac{t}{\tau_i}\right)}\left(\prod\limits_{i=0}^{m}\left(\frac{\tau_i}{t}\right)^{\frac{t}{\tau_i}}\right)^{p^+-1}\left\|(u_k+\delta)\right\|_{t}
	\end{align}
	Again by using the fact $\eta>1$, we have
	$$\sum\limits_{i=0}^{\infty}\frac{t}{\tau_i}=\sum\limits_{i=0}^{\infty}\frac{1}{\eta^i}=\frac{\eta}{\eta-1}$$ and 
	$$\prod\limits_{i=0}^{\infty}\left(\left(\frac{\tau_i}{t} \right)^{\frac{t}{\tau_i}}\right)^{p^+-1}=\eta^{\frac{\eta}{(\eta-1)^2}}.$$
	Therefore, from \eqref{bound est 9}, we get
	\begin{equation}\label{bound est 10}
	\left\|u_k\right\|_{\infty}\leq\left(C|\Omega|^{\frac{p^+}{r^*}-\frac{1}{t}}\right)^{\frac{\eta}{\eta-1}}\left(C'\eta^{\frac{\eta}{(\eta-1)^2}}\right)^{p^+-1}\left\|(u_k+\delta)\right\|_{t}
	\end{equation}
	as $m\rightarrow\infty$ Furthermore, by applying the triangle inequality together with the fact $u_k\leq(u-1)^+$ in \eqref{bound est 10}, we obtain
	\begin{equation}\label{bound est 11}
	\left\|u_k\right\|_{\infty}\leq{C}\left(\eta^{\frac{\eta}{(\eta-1)^2}}\right)^{p^+-1}\left(|\Omega|^{\frac{p^+}{r^*}-\frac{1}{t}}\right)^{\frac{\eta}{\eta-1}}\left(\left\|(u-1)^+\right\|_t+\delta|\Omega|^{\frac{1}{t}}\right)
	\end{equation}
	Finally letting $k\rightarrow\infty$ in \eqref{bound est 11}, we get
	\begin{equation}\label{bound est 12}
	\left\|(u-1)^+\right\|_{\infty}\leq C\left(\eta^{\frac{\eta}{(\eta-1)^2}}\right)^{p^+-1}\left(|\Omega|^{\frac{p^+}{r^*}-\frac{1}{t}}\right)^{\frac{\eta}{\eta-1}}\left(\left\|(u-1)^+\right\|_t+\delta|\Omega|^{\frac{1}{t}}\right)
	\end{equation}
	Hence, we conclude that $u\in L^{\infty}(\bar{\Omega}).$
\end{proof}
\section{Appendix}
\noindent We will begin with the following weak comparison principle.
\begin{theorem}\label{weak comparison}
	Suppose the assumptions \eqref{assm p},\eqref{assm q}, \eqref{assm gamma} and \eqref{assm f} are fulfilled. Let $u, v\in W_0$. If $(-\triangle_{p(x)})^{s}u-\frac{\lambda}{|u|^{\gamma(x)-1}u}\leq(-\triangle_{p(x)})^{s}v-\frac{\lambda}{|v|^{\gamma(x)-1}v}$ weakly in $\Omega$ and $u\leq v$ on $\Omega^{c}$, then $u\leq v$ in $\Omega.$
\end{theorem}
\begin{proof}
	Since $(-\triangle_{p(x)})^{s}u-\frac{\lambda}{|u|^{\gamma(x)-1}u}\leq(-\triangle_{p(x)})^{s}v-\frac{\lambda}{|v|^{\gamma(x)-1}v}$ weakly in $\Omega$ and $u\leq v$ on $\Omega^{c}$, we have
	\begin{align}\label{compprinci}
	\langle(-\Delta_{p(x)})^sv,\phi\rangle-\int_{\Omega}\frac{\lambda\phi}{|v|^{\gamma(x)-1}v}\d x&\geq\langle(-\Delta_{p(x)})^su,\phi\rangle-\int_{\Omega}\frac{\lambda\phi}{|u|^{\gamma(x)-1}u}\d x,~\forall{\phi\geq 0\in W_0.}
	\end{align}
	Now on choosing $\phi=(u-v)^{+}$ as the test function in the inequality \eqref{compprinci}, it follows that
	{\small\begin{align}\label{compprinci1}
		&\langle(-\Delta_{p(x)})^sv-(-\Delta_{p(x)})^su,(u-v)^{+}\rangle-\int_{\Omega}\lambda(u-v)^{+}\left(\frac{1}{|v|^{\gamma(x)-1}v}-\frac{1}{|u|^{\gamma(x)-1}u}\right)\d x\geq 0.
		\end{align}}
	Denote $Q(x, y)=\int_{0}^{1}|(u(x)-u(y))+t((v(x)-v(y))-(u(x)-u(y)))|^{p(x,y)-2}dt.$	Then the identity
	\begin{align}\label{identity}
	|b|^{p(x,y)-2}b-|a|^{p(x,y)-2}a&=(p(x,y)-1)(b-a)\int_0^1|a+t(b-a)|^{p(x,y)-2}\d t
	\end{align}
	with $a=v(x)-v(y)$, $b=u(x)-u(y)$ gives
	\begin{align}
	|u(x)-u(y)|^{p(x,y)-2}(u(x)-u(y))&-|v(x)-v(y)|^{p(x,y)-2}(v(x)-v(y))\nonumber\\
	&=(p(x,y)-1)\{(u(y)-v(y))-(u(x)-v(x))\}Q(x,y).
	\end{align}
	Observe that $Q(x, y)=Q(y, x)\geq 0$ for $(x, y)\in \mathbb{R}^{N}\times \mathbb{R}^{N}$. Furthermore, $Q(x, y)=0$ implies that $v(x)=v(y)$ and $u(x)=u(y)$. Set $\psi(x)=u(x)-v(x)$. Now choose $\phi=(u-v)^{+}$ as the test function. Therefore, on using
	\begin{align}
	\psi&=u-v=(u-v)^{+}-(u-v)^{-}\nonumber
	\end{align}
	we obtain
	\begin{align}\label{negativity}
	[\psi(y)-\psi(x)][\phi(x)-\phi(y)]&=-(\psi^{+}(x)-\psi^{+}(y))^{2}\leq 0.
	\end{align}
	Now from the inequality \eqref{negativity}, it follows that
	$$0\geq\langle(-\Delta_{p(x,y)})^sv-(-\Delta_{p(x,y)})^su,(v-u)^+\rangle
	=-\int_{\mathbb{R}^{N}\times \mathbb{N}^{N}}\!\!\!\!\!\!\frac{(p(x,y)-1)Q(x,y)(\psi^+(x)-\psi^+(y))^2}{|x-y|^{N+sp(x,y)}}\d x\d y\geq 0.$$
	This shows that the Lebesgue measure of $\Omega^{+}$, i.e., $|\Omega^{+}|=0$. In other words $v\geq u$ a.e. in $\Omega$.
\end{proof}
\noindent The next Lemma is a generalization of the well known Lemma A.1 from \cite{Willem1997}.
\begin{lemma}\label{appendix a1 gen}
	Let $\Omega$ be an open subset of $\mathbb{R}^N$ and $1<p^-\leq p(x)\leq p^+<\infty$. If $v_n\rightarrow u$ in $L^{p(x)}(\Omega)$, then there exists a subsequence \{$w_n$\} of \{$v_n$\} and a function $g\in L^{p(x)}(\Omega)$ such that $w_n(x)\rightarrow u(x)$ and $|u(x)|, |w_n(x)|\leq g(x)$ almost everywhere in $\Omega$.
\end{lemma}
\begin{proof}
	Since $v_n\rightarrow u$ in $L^{p(x)}(\Omega)$ then we can assume, upto a subsequence that $v_n(x)\rightarrow u(x)$ almost everywhere in $\Omega$. Therefore there exists a subsequence \{$w_n$\} of \{$v_n$\} such that $\|w_{j+1}-w_j\|_{p(x)}\leq\frac{1}{2^j}$ for all $j\in\mathbb{N}$. Now define, $$g(x)=|w_1(x)|+\sum_{j=1}^{\infty}\|w_{j+1}(x)-w_j(x)\|_{p(x)}.$$
	Thus $|w_n(x)|\leq g(x)$ a.e. in $\Omega$ and hence we get $|u(x)|\leq g(x)$.
\end{proof}
\section*{Acknowledgement}
\noindent The author S. Ghosh, thanks the Council of Scientific and Industrial Research (CSIR), India (09/983(0013)/2017-EMR-I) for the financial assistantship received to carry out this research work.  D. Choudhuri thanks the Science and Engineering Research Board (SERB), India for the research grant (MTR/2018/000525) to carry out the research. R.Kr. Giri acknowledges the financial support and facilities
received from the Mathematics Department, Technion - Israel Institute of Technology.

\section*{References}







\end{document}